\documentclass{amsart}

\usepackage{graphicx}

\newtheorem{theorem}{Theorem} [section]
\newtheorem{proposition}[theorem]{Proposition}
\newtheorem{lemma}[theorem]{Lemma}

\theoremstyle{definition}
\newtheorem{definition}[theorem]{Definition}

\theoremstyle{remark}

\title {Unknotting Tunnels in Hyperbolic 3-Manifolds}

\vspace{1 cm}

\author[Adams]{Colin Adams}
\address{Colin Adams, Department of Mathematics and Statistics, Williams College, Williamstown, MA 01267}
\email{Colin.C.Adams@williams.edu}

\author[Knudson]{Karin Knudson}
\address{Karin Knudson, Department of Mathematics, University of Texas, 1 University Station C1200, Austin, TX 78712 }
\email{kknudson@math.utexas.edu}


\begin{document}

\begin{abstract} An unknotting tunnel in a 3-manifold with boundary is a properly embedded arc, the complement of an open neighborhood of which is a handlebody. A geodesic with endpoints on the cusp boundary of a hyperbolic 3-manifold and perpendicular to the cusp boundary is called a vertical geodesic. Given a vertical geodesic $\alpha$ in a hyperbolic 3-manifold $M$, we  find sufficient conditions for it to be an unknotting tunnel. In particular, if $\alpha$ corresponds to a 4-bracelet, 5-bracelet or 6-bracelet in the universal cover and has short enough length, it must be an unknotting tunnel. Furthermore, we consider a vertical geodesic $\alpha$ that satisfies the elder sibling property, which means that in the universal cover, every horoball except the one centered at $\infty$ is connected to a larger horoball by a lift of $\alpha$. Such an $\alpha$ with length less than $\ln{(2)}$ is then shown to be an unknotting tunnel. 
\end{abstract}

\maketitle

\section{Introduction} 

An \textit{unknotting tunnel} $\alpha$ in a manifold $M$ with boundary is  an arc properly embedded in the manifold  such that $M-N(\alpha)$  is a handlebody.  
Given a knot or link $K$  in $S^3$, an arc $\alpha$ that intersects $K$ in its endpoints is said to be an \textit{unknotting tunnel} if it is an unknotting tunnel when restricted to the exterior of $K$.

The \textit{tunnel number} of a manifold  is the least number of properly embedded arcs  such that the complement of an open regular neighborhood of the arcs is a handlebody.  Every compact orientable manifold with boundary has a finite tunnel number associated with it, but here, we will be dealing with manifolds of tunnel number one. 

We consider finite volume orientable hyperbolic manifolds with one or two cusps. Examples include hyperbolic knot or link complements. If $M$ is such a manifold, then there exists a projection map $p:H^3 \to M$ that generates the hyperbolic structure. Cusps then lift to collections of horoballs in hyperbolic 3-space, and geodesics with both ends going out the cusps lift to collections of geodesics connecting horoballs. Such geodesics are called {\em vertical geodesics} and are candidates for being unknotting tunnels when the interiors of disjoint cups are removed.

In \cite{adams95}, it was proved that for two-cusped hyperbolic 3-manifolds, all unknotting tunnels are vertical geodesics, and further if the length of such a geodesic is defined as the length outside a maximal cusp or set of maximal cusps with disjoint interiors, then the length of a geodesic corresponding to an unknotting tunnel is less than $\ln {(4)}$. 

For one-cusped manifolds, it is not generally known that an unknotting tunnel must be isotopic to a geodesic.  However, in  \cite {AR}, it was shown to be true for hyperbolic 2-bridge knots . In \cite{CFP}, it was proved that an unknotting tunnel in a one-cusped hyperbolic manifold coming from``generic" surgery on a 2-cusped manifold is isotopic to a geodesic. Moreover, in \cite{CLP}, it was proved that unknotting tunnels in one-cusped hyperbolic 3-manifolds can be arbitrarily long, unlike the case for 2-cusped manifolds. 

Given a vertical geodesic $\alpha$ in the manifold, we  find sufficient conditions for it to be an unknotting tunnel. In particular, in the universal cover $H^3$, define an $n$-bracelet to be a cycle of  $n$ horoballs covering the cusps,  connected by lifts of $\alpha$.  We show that if $\alpha$ has length less than $\ln{(\sqrt{2})}$ and possesses a 4-bracelet, it must be an unknotting tunnel. If $\alpha$ has length less than 0.168474, and possesses a 5-bracelet, then it must be an unknotting tunnel. If $\alpha$ has length 0, meaning it corresponds to a point of tangency of the maximal cusp or cusps, it is an unknotting tunnel whenever there is a bracelet of six or fewer horoballs.

We then make use of the elder sibling property to obtain additional sufficient conditions for a geodesic in a knot complement to be an unknotting tunnel. The elder sibling property for balls in hyperbolic space was introduced by Freedman and McMullen in  \cite{FM}, where it was used to create a criterion for tameness of 3-manifolds.  A 3-manifold is tame if it is homeomorphic to the interior of a compact  manifold with boundary.  In 1974 \cite{Marden}, Marden  first raised the question of whether every complete hyperbolic three-manifold with finitely generated fundamental group is topologically tame, and this question became known as the Tameness Conjecture.  Marden proved that geometrically finite hyperbolic three manifolds are topologially tame. Freedman and McMullen developed the concept of elder sibling to make further inroads on the problem, proving with a Morse theory argument that the elder sibling property for  a hyperbolic 3-manifold with finitely generated fundamental group implies tameness.

The Tameness Conjecture was eventually proved in full generality in 2004 by Ian Agol \cite{agol}, and also by Danny Calegari and David Gabai  \cite{CG}.  A number of results follow from tameness, including the Ahlfors Cojecture, posed in the early 1960s, which states that the limit set of a finitely generated Kleinian group is either the whole sphere, or of measure zero \cite{canary08}.

Freedman and McMullen define the elder sibling property to hold for a collection of open balls in $H^3$ if there exists a ball $B_1$ in the set such that  any ball in the collection is joined to $B_1$ by a finite chain of overlapping balls moving monotonically closer to $B_1$.  

We apply a version of elder sibling to collections of horoballs and beams connecting the horoballs corresponding to a cusped hyperbolic 3-manifold and geodesic pair,  and show that if the ball-and-beam pattern is elder sibling, the geodesic must be an unknotting tunnel. We extend this to a weaker condition called almost elder sibling.

Note that we are only considering orientable manifolds throughout, so any mention of manifolds should be taken to mean orientable manifolds, even when this is not specified.

\section{Ball-and-Beam Patterns}
To prove the following results, we  make use of the \textit{ball-and-beam pattern} associated with a given manifold and vertical geodesic pair $(M,\alpha)$, as discussed in \cite{adams95}.  A \textit{vertical geodesic} is a geodesic $\alpha$ in a cusped hyperbolic three-manifold $M$ that is perpendicular to the cusp or cusps at each of its ends.  When the manifold is lifted to the upper-half-space model of $H^3$ so that one end of a lift of the geodesic touches the horoball centered at $\infty$, the lift is  a vertical  ray.   

 The \textit{ball-and-beam pattern} associated with $(M,\alpha)$ is the subset of  $H^3$ given by $p^{-1}(C \cup N(\alpha))$, where $C$ is a  cusp or the union of disjoint cusps. The beams are given by the preimage of $N(\alpha)$. We only consider a beam connected to a ball if its endpoint is at the center of the ball (which in the case of a horoball is the point of tangency of the ball with $\partial H^3$).  It is possible that a beam may intersect a ball that is not one of the two balls at which it has its endpoints, but we consider these intersections to be \textit{ghost intersections}, and they do not count as true intersections when we consider the connectedness of a ball-and-beam pattern.  One can always shrink back the cusp or cusps so that the only intersections of beams with balls occurs when the balls are at the endpoints of the beams.
 
 However, it is also convenient to be able to expand the balls. A  \textit{maximal cusp} of a manifold is a cusp that has been expanded until it first touches itself on the boundary.  A maximal cusp then lifts to a set of horoballs in $H^3$, some of them tangent to each other, with disjoint interiors. In the case of multiple cusps, a maximal cusp collection is any choice of expanded cusps such that none overlap in their interiors, either with themselves or each other, and none can be expanded while preserving this fact. We define the \textit{length} of a vertical geodesic to be the length in $H^3$ of that part of the geodesic that lies between the two points on the geodesic where it intersects the boundary of the maximal cusp or cusps(again excluding ghost intersections). 

To distinguish between horoballs in a horoball pattern, we refer to a particular horoball as $H_a$, where $a$ is the point in the x-y plane at which the horoball is tangent.  The horoball centered at $\infty$ will be denoted $H_\infty$. Most of the horoball patterns depicted in this paper correspond to maximal cusps.

\begin{figure}[ht]
\centering
\includegraphics*[viewport=0 0 500 140]{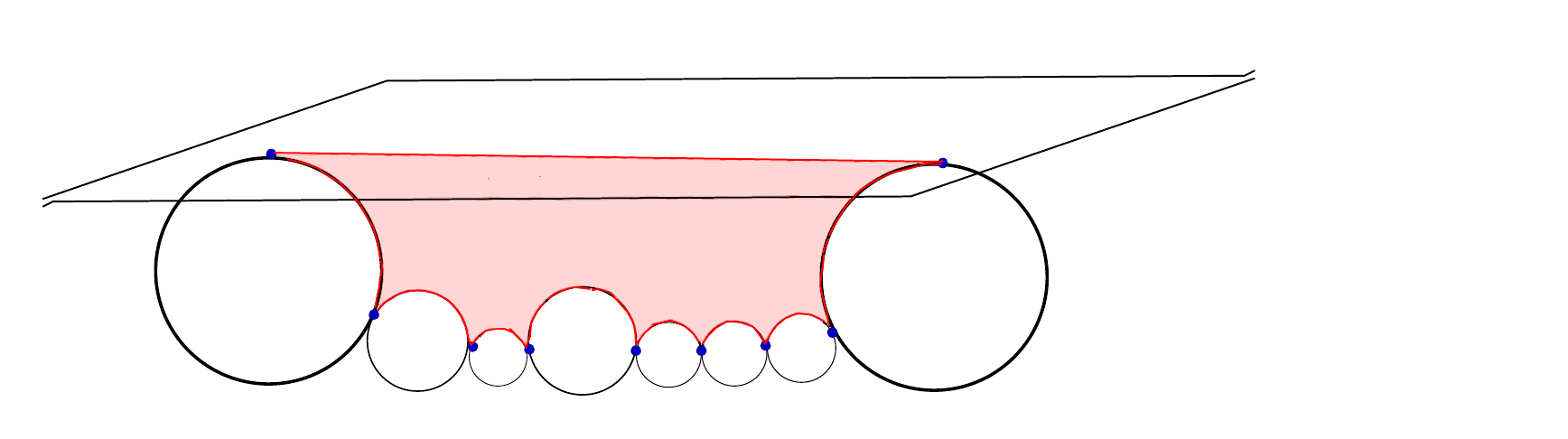}
\caption{Here, $\alpha$ is a vertical geodesic of length $0$, the lifts of which appear as dots corresponding to points of tangency of the horoballs.  The nine horoballs form a 9-bracelet.  The shaded region forms a 9-disk.}
\label{ndisk}
\end{figure}

It  will also be useful to consider $n$-bracelets and $n$-disks in ball-and-beam patterns.  An $n$-\textit{bracelet} is a sequence of $n$ horoballs cyclically connected by beams.  We assume $n \geq 3$. We say that a ball-and-beam pattern contains an $n$-\textit{disk} if there is a disk $D$ in $H^3$ that intersects the ball-and-beam pattern of $(M, \alpha)$ in $\partial{D}$ such that $\partial{D}$ is a nontrivial curve in an $n$-bracelet. See Figure \ref{ndisk} for an example of how an n-bracelet might look.  We say the $n$-bracelet is \textit{blocked} if it does not correspond to an $n$-disk, i.e. there is no nontrivial curve in the $n$-bracelet that bounds a disk with interior in the complement of the balls and beams (Figure \ref{blockeddisk}).  

\begin{figure}[ht]
\centering
\scalebox{.7}{\includegraphics*{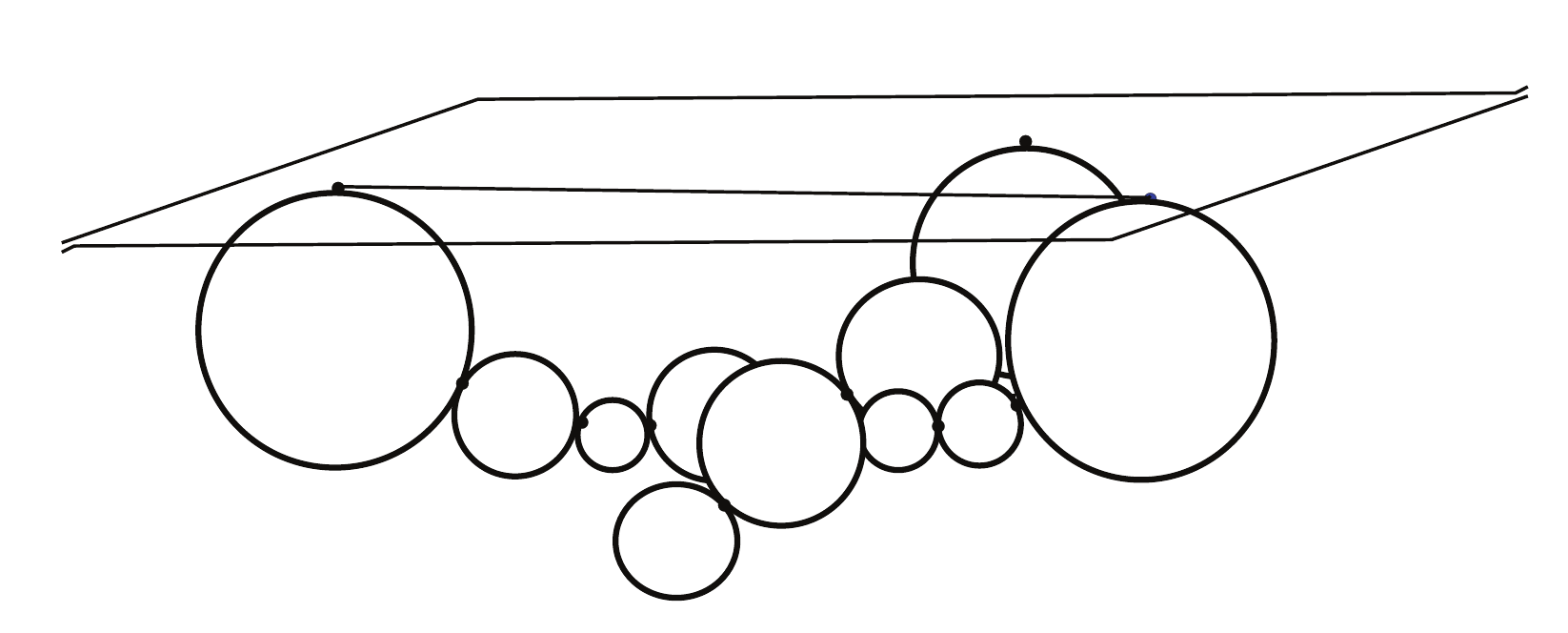}}
\caption{A blocked $n$-disk.}
\label{blockeddisk}
\end{figure}

In fact, a result from \cite{adams95}  renders these $n$-bracelets extremely useful for us:

\begin{lemma} [ \cite{adams95}, Corollary 4.2]
If M is a one-cusped hyperbolic 3-manifold and $\alpha$ a vertical geodesic within it or if M is a 2-cusped hyperbolic 3-manifold and $\alpha$ is a vertical geodesic that has ends at both cusps, and if the ball-and-beam pattern for $(M, \alpha)$ contains an n-disk, then $\alpha$ must be an unknotting tunnel for $M$.
\end{lemma}

In the following section, we use $n$-disks and ball-and-beam patterns to provide sufficient conditions for a given vertical geodesic to be an unknotting tunnel.   But first, we need a few geometric lemmas.

\begin {lemma}\label{lengthofg}
Let $H_1$ and $H_2$ be horoballs of Euclidean radii $r_1$ and $r_2$ centered at points $x_1$ and $x_2$ on the x-y plane, respectively, and let $\gamma$ be the geodesic that runs from $x_1$ to $x_2$.   Let $b$  be the Euclidean radius of the semicircular arc formed by $\gamma$, and let  the segment of $\gamma$ that runs between $H_1$ and $H_2$ have length $g$ (see Figure \ref{twohoroballs}). Then $g = \ln { (\frac{b^2}{r_1 r_2} )}$.
\end{lemma}

\begin{figure}[ht]
\centering
\includegraphics*[viewport=0 0 250 150]{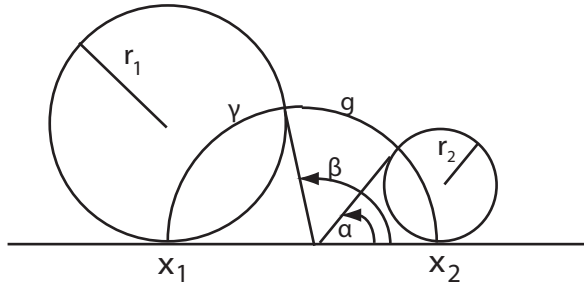}
\caption{The situation described in Lemma 2.2.}
\label{twohoroballs}
\end{figure}

\begin{proof}
The hyperbolic length of the geodesic arc connecting $H_1$ and $H_2$ is $\ln{( \frac{ csc\alpha - cot \beta}{csc \alpha - cot \alpha} )} = \ln{(\frac {(1 - cos\beta)(sin\alpha) }{(1-cos \alpha)(sin \beta) } ) }$, where $\alpha$ and $\beta$ are the angles shown in Figure \ref{twohoroballs}.  Calculating the point of intersection of $H_1$ and the semicircular arc of radius $b$, we can show that $\cos{\alpha} = \frac{ b^2- r_1^2} { b^2 + {r_1}^2 } $ and $\sin{\alpha} = \frac {2br_1}{b^2+r_1^2} $.  Similarly, $\cos{\beta} = \frac{ b^2- r_2^2} { b^2 + {r_2}^2 } $ and $\sin{\alpha} = \frac {2br_2}{b^2+r_2^2} $.  Substituting these values for $\sin{\alpha}$, $\cos{\alpha}$, $\sin{\beta}$, and $\cos{\beta}$ into the above formula for $g$ and simplifying yields that $g = \ln{(\frac{b^2}{r_1 r_2} )}$.

\end{proof}

\begin{lemma}
For two horoballs $H_1$ and $H_2$ of radius $r_1$ and $r_2$ centered at points $x_1$ and $x_2$, respectively, on the xy-plane and connected by a geodesic arc segment of length $g$, the distance $d(x_1,x_2)$ between their points of tangency with the xy-plane is equal to $2 \sqrt{r_1 r_2 e^g } $.
\end{lemma}

\begin {proof}
The distance between the two horoballs is the twice the Euclidean radius $b$ of the geodesic arc that runs from $H_1$ to $H_2$.  By the lemma above, $g = \ln{(\frac{b^2}{r_1 r_2} )}$, and so $d(x_1, x_2) = 2b = 2 \sqrt{r_1 r_2 e^g } $.

\end {proof}

We  now use these lemmas to provide conditions that ensure that a bracelet of length 4, 5 or 6  in a ball-and-beam pattern for a  manifold of one or two cusps corresponds to an $n$-disk. Note that we need not consider the case where a ball-and-beam pattern contains a 3-disk.  By Lemma 5.1 in \cite{adams95}, given a noncompact, orientable hyperbolic 3-manifold M and vertical geodesic $\alpha$, the ball-and-beam pattern for $(M, \alpha)$ cannot contain a 3-disk.

\begin{proposition}\label{prop:4}
Let $M$ be a hyperbolic 3-manifold of one or two cusps and let $\alpha$ be a vertical geodesic of length less than $\ln{(\sqrt{2})}$.  If M has two cusps, we require that $\alpha$ run from one cusp to another. Then if the ball-and-beam pattern of $(M, \alpha)$ contains a 4- bracelet, $\alpha$ must be an unknotting tunnel.
\end{proposition}

\begin{proof}
Suppose we have a  hyperbolic manifold of one or two cusps with vertical geodesic $\alpha$ such that the hyperbolic length of $\alpha$ in $H^3$ is $g$.   We will show that for small $g$, the wrist hole of a 4-bracelet becomes ``too small"  to be blocked by other balls and beams.  

We consider a 4-bracelet that has $H_\infty$ as one of its horoballs, since we can always apply an isometry to make this the case. For convenience, we label the two horoballs in the 4-bracelet connected to $H_\infty$ as $H_a$ and $H_b$, and the ball in between them $H_c$.  Note that since $H_a$ and $H_b$ are both connected to $H_\infty$ by a segment of $\alpha$ with length $g$, they are the same distance from $H_\infty$ and thus have the same Euclidean height.  Let the two blocking balls be labeled $H_e$ and $H_f$, and note that in order to block, they must be connected with a copy of the vertical geodesic that also has length $g$.  We are looking for the shortest value for $g$ for which this 4-bracelet can be blocked, so we can consider $H_e$ and $H_f$ to be equal in size, since decreasing the size of one of these balls to make one smaller than the other never allows a bracelet to be blocked that could not have been blocked before by the two equal-sized balls.   

Because $H_c$ is connected to $H_a$ and $H_b$ by beams of the same length, it must be that $c$, the point of tangency with the x-y plane of  $H_c$, is equidistant from $a$ and $b$.  Moreover, we can assume that $c$ is in line with $a$ and $b$, since moving it out of line with $a$ and $b$ while maintaining $g$ forces $a$ and $b$ to come closer together, which forces one of $H_e$ to $H_f$ to be smaller. As mentioned before, decreasing the size of $H_e$ or $H_f$ does not allow any bracelet to be blocked that could be blocked with larger horoballs.

\begin{figure}[ht]
\centering
\includegraphics*[viewport=0 0 230 190]{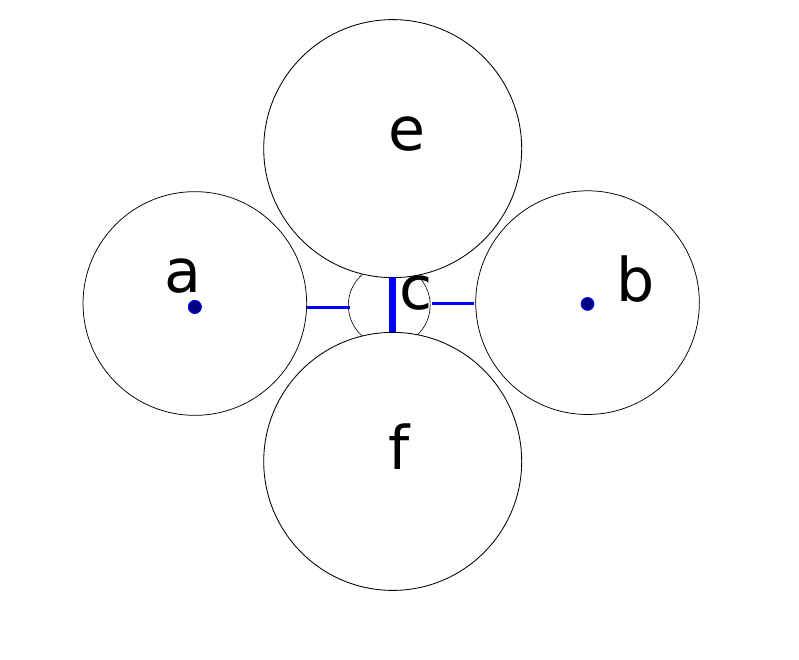}
\caption{A view from above of a blocked 4-disk.}
\label{blockedfour}
\end{figure}

Let $r$ be the radius of $H_e$ and $H_f$, and let  the shortest Euclidean distance between the two horoballs be called $d$. Since the distance in the x-y plane between the centers of $H_e$ and $H_f$ is $2r +d$, we can use Lemma 2.2 above to calculate that $g = \ln{ ( \frac{ (r+ \frac{d}{2})^2}{r^2} ) } = \ln{ (1 + \frac{d}{r} + \frac{d^2}{4r^2} ) } $. 

Next, we can set a bound on the Euclidean radius $r_c$ of $H_c$ for given values of $r$ and $g$.  The largest $H_c$ can be while fitting under the blocking balls $H_e$ and $H_f$ is in the case where it is tangent to both of them.  In this case, it has radius $\frac{r e^g}{4}$, and so it must be true that $r_c \leq \frac{r e^g}{4}$.   

$H_a$ and $H_b$ are hyperbolic distance $g$ from $H_\infty$, which is bounded by a Euclidean plane at height 1, and since the hyperbolic distance of a vertical line segment running from height $z_1$ to height $z_2$ above the x-y plane is given by $\ln({ \frac { z_1 }{ z_2 } )}$, it follows that $H_a$ and $H_b$ have Euclidean height equal to $e^{-g}$ and Euclidean radius $\frac { e^{-g} }{2} $.

Now, we also must ensure that $H_a$ and $H_b$ do not intersect $H_e$ or $H_f$.  Suppose that $H_a$ and $H_b$ are tangent to $H_e$ and $H_f$.  Then the distance from $a$ to $d$ on the x-y plane is $\sqrt{2r e^{-g}} $.  In general, then, $d(a, e) \geq \sqrt{2r e^{-g} }$.  Lastly, we can calculate Euclidean distance from $c$ to $e$ on the x-y plane, since $d(c, e) = \frac{1}{2} d(e, f)$, which by Lemma 1 is equal to $\frac {1}{2} (2 \sqrt{r^2 e^g}) = r \sqrt{e^g}$. 

 Now that we have bounds on the lengths of the line segments $\bar{ac}$, $\bar{ae}$, and $\bar{ec}$ for a blocked 4-disk, we can relate these quantities using the fact that by the Pythagorean Theorem, $d(a,c) ^2 + d(c,e)^2 = d(a,e)^2$, and thus $ \frac {re^g}{2} + r^2 e^g \geq d(a,c) ^2 + d(c,e)^2 = d(a,c)^2 \geq 2r e^{-g}$.  The fact that $ \frac {re^g}{2} + r^2 e^g \geq  2r e^{-g} $ implies that $g \geq \ln {( \sqrt{ \frac {4}{1+ 2r} } ) }$.  Note that as $r$ increases, this lower bound on $g$ decreases.

  However, since $H_e$ and $H_f$ cannot have Euclidean height greater than 1, $r$ cannot be any larger than $\frac{1}{2} $, and thus $g$ cannot be any shorter than $ \ln{( \sqrt {2} )} $.  Hence, for a 4-bracelet to be blocked, $g$ must be at least $\ln{(\sqrt{2})}$.  So if  $g < \ln{(\sqrt{2})}$ and if the ball-and-beam pattern for $(M, \alpha)$ contains a 4-bracelet, then the 4-bracelet bounds a disk, and so $\alpha$ must be an unknotting tunnel.

\end{proof}

\begin{proposition} \label{prop:45}
Let $M$ be a one-cusped hyperbolic 3-manifold and let $\alpha$ be a vertical geodesic of length less than 0.168474.  Then if the ball-and-beam pattern associated with $(M, \alpha)$ contains a 4-bracelet or a 5- bracelet,  $\alpha$ must be an unknotting tunnel.
\end{proposition}

\begin{proof}
If the ball-and-beam pattern for $(M, \alpha)$ contains a 4-bracelet, then the fact that $\alpha$ is an unknotting tunnel follows from Proposition \ref{prop:4}.  All that remains to show is that for sufficiently small vertical geodesic length $g$, it is impossible for a 5-bracelet to be blocked.

We take the 5-bracelet to contain $H_\infty$, since an isometry can always be applied to make this the case.  The other horoballs of the 5-bracelet are labeled $H_a$, $H_b$, $H_c$ and $H_d$, and the blocking balls $H_e$ and $H_f$ as in Figure \ref{blockedfive}.  As before, $H_a$ and $H_b$ are both a hyperbolic distance $g$ from $H_\infty$, and so have the same radius, which we call $r_a$.  Also, making one of $H_e$ and $H_f$ smaller does not allow any bracelet to be blocked that could not already be blocked before, so we take $H_e$ and $H_f$ to be the same size, and call their Euclidean radius $r_e$.  Moreover, we may take the centers a, b, c, and d of horoballs $H_a$, $H_b$, $H_c$, and $H_d$ to be collinear for the following reason.  The Euclidean distances between each of these balls is determined by their radius and the length $g$ of the beams between them.  If the horoballs were arranged not in a straight line, and it was still possible for the 5-bracelet to be blocked, we could ``straighten out" the path in the x-y plane from a to c to d to b, maintaining the radii of each horoball and the distances between their centers.  Since the original bracelet must have been disjoint from the two equal-sized blocking balls $H_e$ and $H_f$, the straightened-out bracelet is also disjoint from the blocking balls $H_e$ and $H_f$, and so the straightened version of the original 5-bracelet can also be blocked.   Thus, we need only consider the case where a, b, c, and d are collinear already.  

We can also assume that $H_c$ and $H_d$ are of the same size.  Suppose there is a 5-bracelet that is blocked, and that $H_d$ is smaller than $H_c$.  If $H_d$ is expanded while $g$ and the Euclidean radius of the other balls are held the same, then the Euclidean distance between $H_d$ and the neighboring balls in the bracelet must increase.  Because increasing this distance does not cause any horoballs to intersect, the process of expanding $H_c$ will create a valid new 5-bracelet with beams of length $g$ that is blocked.  So, to find the shortest length $g$ for which a 5-bracelet can be blocked, we need only consider the case where $H_c$ and $H_d$ are the same size.  We call their Euclidean radius $r_c$.   

In a 5-bracelet with horoballs configured in this way, we let $x$ denote the point on the x-y plane that is midway between $c$ and $d$ and between $e$ and $f$.

\begin{figure}[ht]
\centering
\includegraphics*[viewport=0 0 250 250]{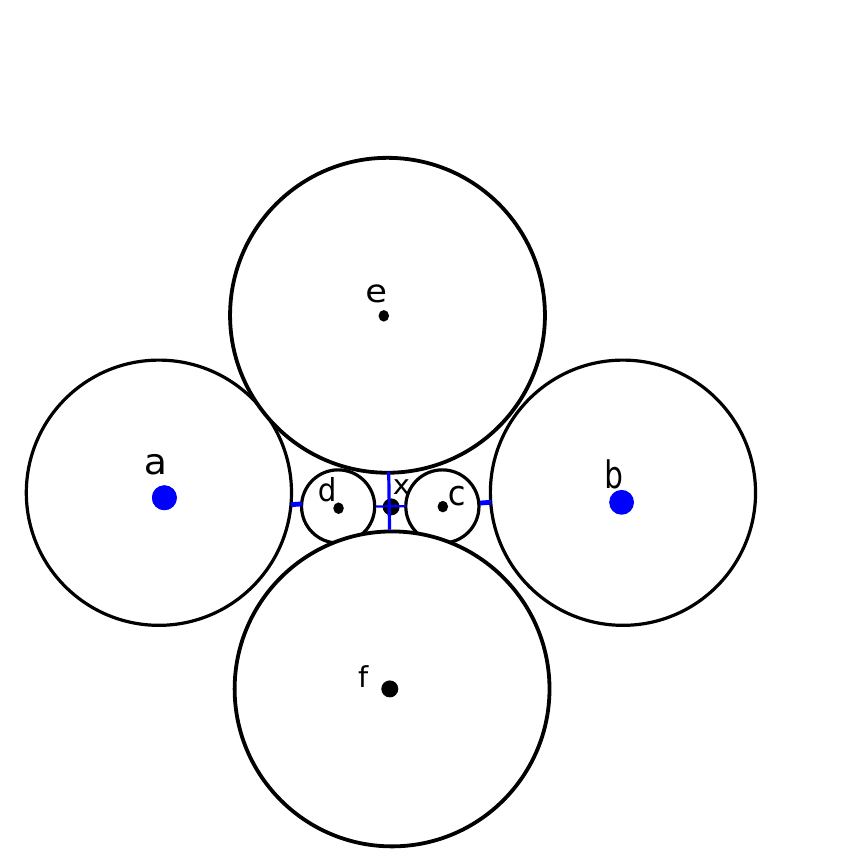}
\caption{A view from above of a blocked 5-bracelet, with the beams corresponding to $\alpha$ marked in blue.}
\label{blockedfive}
\end{figure}

Because $g$ is the hyperbolic length of the line segment connecting the top of $H_a$ and $H_b$ to the boundary of $H_\infty$, which is the plane of height 1, the Euclidean radius of $H_a$ and $H_b$ is $r_a = \frac{e^{-g}}{2}$.   The length of  the line segment $\bar{cx}$ is half the distance between $c$ and $cd$, so applying Corollary 2.3,  $d(c,x) = r_c \sqrt{e^g}$. Similarly, $d(e,x) = r_e \sqrt{e^g}$, and $d(a,c) = 2 \sqrt{r_a r_c e^g} = \sqrt{2r_c}$.  In order for $H_e$ and $H_f$ to be disjoint from $H_c$, and $H_d$, it must be the case that $d(a,e) \geq 2 \sqrt{r_c r_e}$.  Thus $r_c^2 e^g + r_e^2 e^g = d(c,x)^2 + d(e,x)^2 = d(c,e)^2 \geq 4r_c r_e $.  From the fact that  $r_c^2 e^g + r_e^2 e^g  \geq 4r_c r_e $ it follows that $r_c \leq \frac{r_e e^g}{2 + \sqrt{ 4-e^{2g}}}$.

Now, in order for $H_e$ and $H_f$ to be disjoint from $H_a$ and $H_b$, it must be true that $d(a,e) \geq 2\sqrt{r_a}{r_c} = \sqrt{2r_e e^{-g}}$.  So we have that $2r_e^{-g} \leq d(a,e)^2 = d(a,x)^2 + d(e,x)^2 = 2r_c + 2r_c \sqrt{2r_c e^g} + r_c e^g$, and applying the previously obtained inequality, this implies that $2r_e^{-g} \leq r_e^2 e^g + \frac{2r_e e^g}{2 + \sqrt{ 4- e^{2g} } } +\frac{2r_e e^{2g} \sqrt{r_e}} { {({2 +\sqrt{4 - e^2g}})}^{3/2}} + \frac{{r_e}^2 e^3g}{ {( 2+ \sqrt{4-e^{2g}})}^2}$. Since $r_e$ is the radius of a horoball in the cusp diagram, it must be between 0 and $\frac{1}{2}$.  The previous inequality results in a lower bound for $g$ that is strictly decreasing over these values for $r_e$.  Thus, the smallest lower bound for $g$ that could possibly be obtained from these restrictions occurs when $r_e$ is equal to $\frac{1}{2}$ in which case we can compute that $g$ must be at least 0.168474...  Thus, if $g < 0.168474... $ a 5-bracelet cannot be blocked, and so the presence of a 5-bracelet implies that the ball-and-beam pattern of the manifold and geodesic pair $(M, \alpha)$ contains an n-disk and hence $\alpha$ is an unknotting tunnel.
\end{proof}

Although it may seem that perhaps relatively few vertical geodesics  have length $g< 0.168474... $, note that in fact, in every one-cusped manifold there is a vertical geodesic of length 0.  The maximal cusp is obtained by expanding the cusp until it first becomes tangent to itself, and the arc that passes perpendicularly through this point of tangency with endpoints perpendicular to the two cusps is a vertical geodesic of length 0.  Moreover, 5-bracelets do arise in ball-and-beam patterns for orientable manifolds, as in this example from the horoball pattern for the (-2, 3,7) pretzel knot (Figure \ref{5bracelethoroball}) where we take the vertical geodesic of length 0 corresponding to tangency points of the horoballs.

\begin{figure}[ht]
\centering
\includegraphics*[viewport=10  10 100 100]{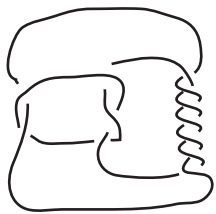}
\caption{The (-2, 3, 7) pretzel knot.}
\label{237}
\end{figure}

\begin{figure}[ht]
\centering
\includegraphics*[viewport=90 80 300 290]{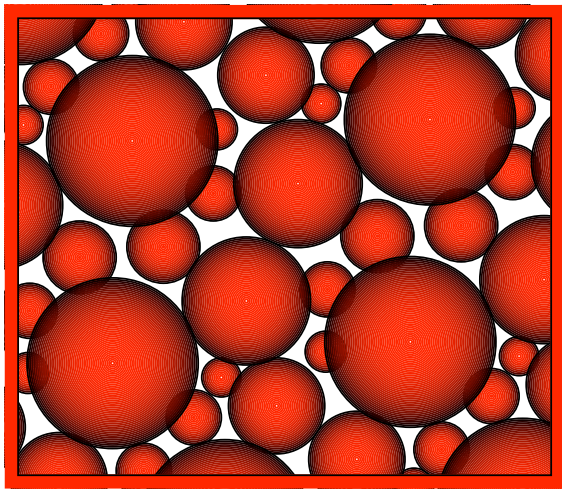}
\caption{A section of the horoball pattern for the (-2, 3,7) pretzel knot produced by \cite{snappea}.  When the proper length-0 vertical geodesic is added to make this a ball-and-beam pattern, the sequence of the four tangent horoballs running from the top right to the bottom left in this illustration, together with $H_\infty$ form an unblocked 5-bracelet that corresponds to a 5-disk.}
\label{5bracelethoroball}
\end{figure}

We conclude by noting that the same type of argument does not work to find a value for $g$ below which an $n$-bracelet in the ball-and-beam pattern must correspond to an $n$-disk, for $n$ equal to 6 or greater.   Consider Figure \ref{sixbracelet1} where we see an arrangement of horoballs and beams corresponding to a  geodesic of length zero, where a 6-bracelet defined by the hexagon is  blocked by the geodesic with endpoints at the center horoball and the horoball at infinity and thus does not correspond to an n-disk.


However, one can show that in fact, the configuration depicted cannot occur for a hyperbolic 3-manifold of one or two cusps.

\begin{proposition}
Let $M$ be a  hyperbolic 3-manifold of one or two cusps and let $\alpha$ be a vertical geodesic of length $0$ that connects the two distinct cusps if $M$ is two-cusped.  Then if the ball-and-beam pattern associated with $(M, \alpha)$ contains a bracelet of length 4, 5, or 6,  then $\alpha$ must be an unknotting tunnel.
\end{proposition}

\begin{proof} By Propositions \ref{prop:4} and \ref{prop:45}, we need only consider bracelets of length 6. We  show that  the local situation that must occur for a bracelet of length six to be blocked when $\alpha$ has length 0 cannot occur.
   
 Suppose a bracelet of length six is blocked. Then choosing one of the blocking horoballs to be the horoball at infinity, we obtain a picture as in Figure \ref{sixbracelet1} where each of the depicted edges and the vertical edges are all in the same edge class. We assume the vertical edge points up out of the xy-plane.Considering the orientations on the edges in the bracelet, one finds that up to rotation and reflection, there are nine possibilities, listed in Figure \ref{sixbracelet2}. 
 
 \begin{figure}[ht]
\centering
\includegraphics*[viewport=130 450 400 700]{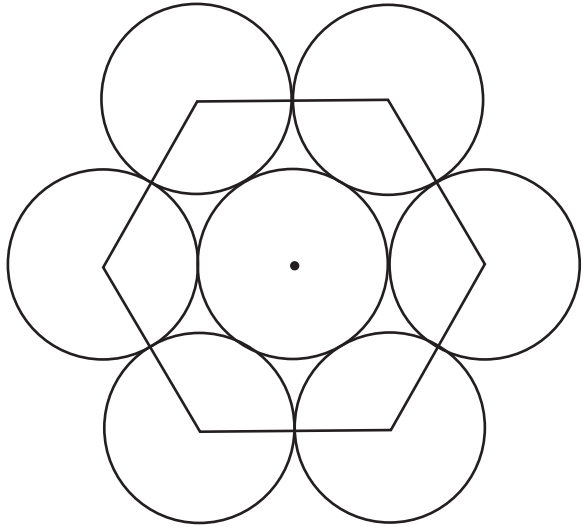}
\caption{A blocked six-bracelet when $g=0$.}
\label{sixbracelet1}
\end{figure}

\begin{figure}[ht]
\centering
\includegraphics*[viewport=60 430 400 700]{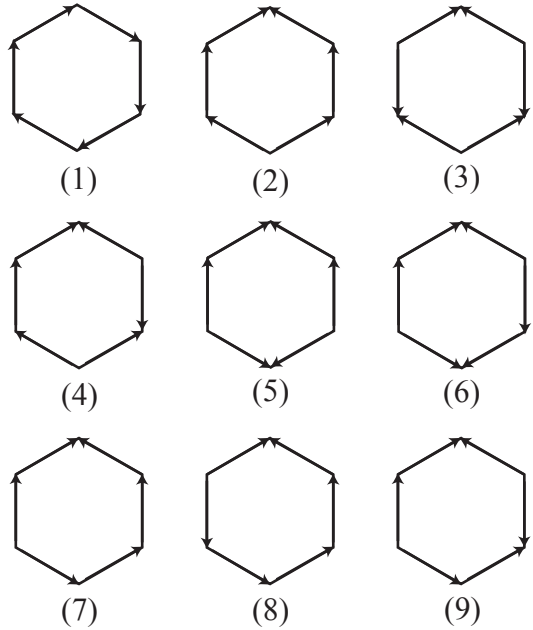}
\caption{Possible orientations on the 6-bracelet edges.}
\label{sixbracelet2}
\end{figure}
 
    We normalize so that the diameter of the horoballs depicted is 1. In general we seek a contradiction by showing that there is an isometry in the fundamental group that fixes points, a contradiction.  
    
    If there are two arrows with heads or tails that meet,  as occurs in all but case (1), there must be two arrows with tails or heads that meet. This meeting at 120 degrees means that from the point of view of infinity, there must be two vertical parallelly oriented edges a distance $\sqrt{3}$ apart. Since any two such arrows must be related through a parabolic isometry that is a Euclidean translation, this can only occur when there are two opposite edges on the hexagon with parallel orientations. This eliminates Type (6) from consideration.
    
    In the cases of (5), (8), and (9), there is only one pair of parallelly oriented opposite edges so there must be a parabolic isometry fixing ${\infty}$ that identifies these two edges. Applying this parabolic isometry to the original hexagon gives a second hexagon attached to the first. In each case, there are two junctures of three arrows. For (5), there is a juncture with one arrow head and two arrow tails. For (9), there is a juncture with two arrow heads and one arrow tail. In both cases, this implies there must be a downward pointing vertical arrow a distance $\sqrt{3}$ from the upward pointing arrows. But the downward pointing arrow must have a flipped version of the hexagon around it, and there is no way to match the arrows on the flipped hexagon with the arrows on the existing hexagons to put the downward pointing arrow in the correct position relative to the upward pointing arrows. So this eliminates (5) and (9) from consideration.
    
    To eliminate the remaining cases, we will also consider the fact that each hexagon is the projection of a region in hyperbolic space that can be subdivided into six vertical ideal regular tetrahedra, the edges of which all pass through points of tangency of the eight horoballs depicted (including $H_\infty$). The single labelled edge class includes 12 of the edges of these tetrahedra, which together subtend an angle of 720 degrees. Hence some of these six tetrahedra must be identified with one another, and  they can generate at most three tetrahedral equivalence classes. 
    
Note that no two tetrahedra can be identified by an identification that sends the central edge of one to the central edge of another, as then that edge would be a fixed point set for the isometry, contradicting the fact all isometries in the fundamental group must be fixed point free.  Define a standard identification to be when tetrahedron  $T_1$ is identified to $T_2$ with an orientation preserving isometry such that the  central and outer edges of $T_1$ go to the outer and central edges of $T_2$ respectively.

      Suppose that there is a non-standard identification between two tetrahedra $T_1$ and $T_2$. Then a labelled edge on $T_1$ must go to an unlabelled edge on $T_2$. Since the labelled edges are opposite pairs, both of the labelled edges on $T_1$ must go to unlabelled edges. Hence, there are four edges in this equivalence class on each of the tetrahedra. Since, after tetrahedral identification, there can be at most six such edges in the edge class, this implies that there can be only two classes of tetrahedra, one with all tetrahedra having four edges in the edge class, which does not include an opposite pair, and one with all tetrahedral having two edges in the edge class that are an opposite pair. Call a representative of each tetrahedral class $A$ and $B$ respectively.  Then all faces of $A$ have two edges in the equivalence class, whereas no face of $B$ has two edges in the equivalence class. Hence, there is no way to glue faces of A to faces of B without adding edges to the edge class, a contradiction to the fact there are at most six edges in the edge class after identifications. Hence this cannot occur, and all identifications of tetrahedra must be standard.

    If two of the tetrahedra are identified, their outer edges must have the same orientation, either clockwise or counterclockwise on the hexagon, since the identification of $T_1$ with $T_2$  must send the outer edge of $T_1$ to the central edge of $T_2$ and the central edge of $T_1$ to the outer edge of $T_2$.

For Case (8), we know that the two opposite parallelly oriented edges are identified by a parabolic isometry $p$ fixing $\infty$. Label the tetrahedra as in Figure \ref{sixbracelet3}. Since one outer edge is oriented one way, while the other five are oriented the other, tetrahedron $A_1$ cannot be identified to any others. Since the parabolic isometry puts a copy of $A_1$ to the right of $B_3$, and no other tetrahedron can be identified to $A_1$, the only tetrahedron that could be identified to $B_3$ is $B_1$, as the isometry $f$ would also send a copy of $A_1$ to $p(A_1)$. However, then $p^{-1} f(A_1) = A_1$ and $p^{-1} f$ is a nontrivial element of the group of isometries with a fixed point, a contradiction. 

     So both $A_1$ and $B_3$ must be their own equivalence classes, which implies that the other four tetrahedra are identified with one another. But if $B_1$ is identified to $B_2$ by $g$, then $g(A_1)$ will be adjacent to $B_2$. If  $B_5$ is identified to $B_2$ by $h$, then $h(B_4)$ is adjacent to $B_2$ and is therefore identified with $g(A_1)$, contradicting the fact $A_1$ is the only tetrahedron in its equivalence class. So this case cannot occur.
    
 Consider cases (4) and (7). In both cases, two outer edges are oriented one way and the four others are oriented the opposite direction. Name the six tetrahedra $A_1, A_2$ and $B_1, B_2,B_3, B_4$ accordingly, as in Figure \ref{sixbracelet3}. We consider the case where $A_1$ and $A_2$  are not identified or are identified.
    
    \begin{figure}[ht]
\centering
\includegraphics*[viewport=20 500 400 610]{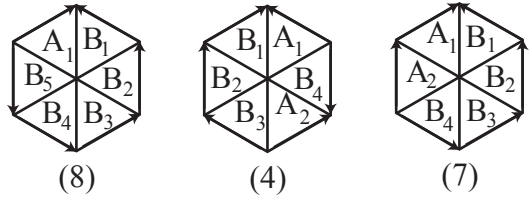}
\caption{Labelling tetrahedra for cases (8),  (4) and (7).}
\label{sixbracelet3}
\end{figure}
    
    If $A_1$ and $A_2$ are not identified, then they form two of the three final tetrahedral classes we can have. So all four of $B_1, B_2,B_3$,  and $B_4$ must be identified. But in both cases, the outer edge of $B_1$ shares an arrowhead with the outer edge of $A_1$ while the arrowhead of the outer edge of $B_3$ meets the tail end of the outer edge of $B_2$. When $B_1$ and $B_3$ are identified to $B_2$, these adjacent outer edges are sent to the same vertical edge but with opposite orientations, a contadiction. 
        
    If $A_1$ and $A_2$  are identified, we consider (4) first. Let $g$ be an isometry sending $A_1$ to $A_2$. Then $g$ sends $B_1$ to a vertical tetrahedron adjacent to $A_2$, with its far vertical edge pointed up. So there must be a parabolic translation $p$ that identifies it with the central edge of the hexagon. But this isometry will identify $g(B_2)$ with $B_2$. So $pg$ sends $B_2$ back to itself but is not the identity. So it is an element of the fundamental group that fixes points, a contradiction.
    
    For Case (7) when an isometry $g$ identifies $A_1$  to  $A_2$, $g(B_1)$ is adjacent to $A_2$. Then there exists a vertical edge a distance $\sqrt(3)$ to the left of the central edge. So there must be a horizontal parabolic isometry $p$ identifying these two edges.  Then $p(g(B_1))$ is identified with $B_2$. Then neither $B_3$ nor $B_4$ can be identified with $B_2$ since such an identification would place a downward pointing vertical edge a distance $\sqrt(3)$ to the right of the central edge, where we already have an upward pointing vertical edge.
    
    However, $(pg)^{-1}(B_2) = B_1$ which means $B_1$ is sent to a tetrahedron adjacent to itself by $(pg)^{-1}$. In particular, as in Figure \ref{sixbracelet4}, a flipped hexagon must be centered at this edge. This forces $B_1$ to be identified to $B_4$, a contradiction.
    
        \begin{figure}[ht]
\centering
\includegraphics*[viewport=30 500 400 650]{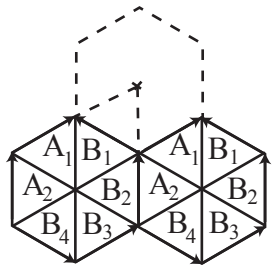}
\caption{Pattern for (7).}
\label{sixbracelet4}
\end{figure}

    This eliminates Cases (4) and (7).

    For Cases (2) and (3), there are three outer edges oriented one way and three oriented the other. Since we can have at most three tetrahedra after identification, this implies that three of the tetrahedra must be identified. Call them $A_1, A_2$ and $A_3$. 
    
    Consider Case (2). Assuming $A_2$ to be the middle one on the hexagon, the identification g of $A_1$ to $A_2$ sends $A_2$ to a vertical tetrahedron adjacent to $A_2$ with its far edge a vertical edge pointed down. The identification f of $A_3$ to $A_2$ sends the fourth tetrahedron adjacent to $A_3$ to a vertical tetrahedron adjacent to $A_2$ with its far edge pointed up. This is a contradiction.
  
      For Case (3), the parabolic isometry that identifies one edge of the hexagon with the opposite edge places a copy of the hexagon next t itself and forces there to be three arrowheads meeting pairwise at 120 degrees and three arrow tails meeting at 120 degrees. Hence there is a second parabolic that places a third copy of the hexagon so that their three centers form an equilateral triangle. Two copies of this triangle form a fundamental domain for the action of the parabolic subgroup on the plane. But there must be a downward oriented vertical edge within this fundamental domain and it must pass through a tangency point of horoballs, so it must be at the center of one of the two equilateral triangles making up this fundamental domain. However, then six more edges, one per tetrahedron, are in this same equivalence class. Hence, it must be the case that there are only two equivalence classes of tetrahedra, one consisitng of those with clockwise orientation on their outer edge and one consisting of those with counterclockwise orientation on their outer edge. However, when two tetrahedra are identified by a standard identification, an unlabelled edge is identified to a newly labelled edge, meaning there are even more eges in the dge class, and implying that this case cannot occur.

     For Case (1), the fact that there is a tail of an arrow touching the head of an arrow at a 120 degree angle implies that there must be a vertical copy of this edge coming out of the xy-plane a distance of exactly  $\sqrt{3}$ from a copy of this edge going into the plane. Hence there must be a flipped copy of the hexagon sharing an edge with the original hexagon.  However, now there are two heads of arrows meeting at an angle of 120 degrees. This means that there are two vertical edges oriented out of the xy-plane that are a distance $\sqrt{3}$ apart. However, any two vertical edges in the same class with the same orientation must be identified by a translational parabolic isometry.  This will cause two edges of the original hexagon to be identified in a manner that does not match their orientations, a contradiction.
     
\end{proof}

Note that a vertical geodesic $\alpha$ of length 0 can have a blocked 6-cycle of exactly the type described above if we allow the manifold to be nonorientable. See the manifold m025 from the cusped census of SNAPPEA \cite{snappea}, which consists of exactly three regular tetrahedra.

\section{The Elder Sibling Property}
Next, we present the elder sibling and almost elder sibling properties for a ball-and-beam pattern, which provide additional sufficient conditions for a vertical geodesic to be an unknotting tunnel.

\begin{definition}
A ball-and-beam pattern is said to be \textit{elder sibling} if every horoball $H_a$ in the pattern is connected to $H_\infty$ by an alternating sequence of horoballs and beams such that all of the balls in the chain have Euclidean radius greater than $H_a$.  Such a chain of balls and beams will be called an \textit{elder sibling chain} of $H_a$. 
\end{definition}

To better understand the properties of an elder sibling ball-and-beam pattern, we note the following.

\begin{lemma} \cite{adams95} In a ball-and-beam pattern, there are only finitely many Euclidean sizes of horoballs greater than or equal to a given value. 
\label{sizelemma}
\end{lemma}

\begin{proof}
The fundamental domain for the cusp is formed by a parallelogram on the x-y plane of finite area. But there is not room inside it for infinitely many balls tangent to the x-y plane of size greater than or equal to a given value.  So there can be only finitely many horoballs of Euclidean radius greater than or equal to a given value in the fundamental domain, and hence there are only finitely many sizes of horoballs in the ball-and-beam pattern larger than a given size.

\end{proof}

Notice that our definition of elder sibling is equivalent to stating that every horoball must be connected to $H_\infty$ by a sequence of balls and beams such that the Euclidean radius of the balls in the sequence is strictly increasing.  Moreover, by Lemma \ref{sizelemma}, these chains of horoballs of increasing size must be finite.

Also, note that if a ball-and-beam pattern is elder sibling, it must be connected, since the fact that every ball can be joined to $H_\infty$ implies that any two balls can be joined by a path that goes from one ball to $H_\infty$ and then from $H_\infty$ to the other ball.  This fact will become useful with the help of an observation from \cite{adams95}.

\begin{lemma} \cite{adams95}
If a ball-and-beam-pattern is connected, it must contain an $n$-bracelet.
\label{connectedimpliesnbracelet}
\end{lemma}

To formulate the next lemma, we suppose that an n-bracelet is blocked. Define a {\it pair of blocking balls} to be any pair of  horoballs in the ball-and beam pattern connected by a beam that punctures a disk with boundary a nontrivial curve in the n-bracelet.  

Intuitively, it would seem that in the ball-and-beam pattern for a manifold and vertical geodesic pair $(M, \alpha)$, in order that an $n$-bracelet is blocked, the blocking balls need to be larger than a certain size in order to reach over the balls of the blocked $n$-bracelet.  In fact, this is the case, and the size of these blocking horoballs will be related to the length of $\alpha$ and the size of the smallest ball of the blocked bracelet in the following way.

\begin{lemma} Let $(M,\alpha)$ be a hyperbolic manifold and vertical geodesic pair, where $M$ has one or two cusps and if it has two cusps, $\alpha$ connects them. Let $\alpha$ have length $g < ln(2)$.
Then if there exists an $n$-bracelet in the ball-and-beam pattern containing $H_\infty$ and it  is blocked, there is a blocking ball in any blocking ball pair of radius at least $\frac{2+\sqrt{4-e^{2g}}}{e^g}$ times the Euclidean radius  of the smallest ball in the n-bracelet.
\end{lemma}

\begin{proof} Suppose we have a blocked $n$-bracelet, with the smallest ball in the bracelet having Euclidean radius $r$. The beam connecting any pair of blocking balls must reach a sufficient vertical height to pass over the n-bracelet. Let $H_e$ and $H_f$ be a pair of blocking balls for the bracelet. We can assume that they are the same size since if the first were larger than the second, expanding the second while preserving $g$ only increases the vertical height attained by the beam. Let $d$ be the Euclidean distance between their centers and $R$ their radius. Then the line segment between their centers must intersect one of the line segment projections of the beams from the $n$-bracelet. We can assume that it intersects the projection of one of the beams leaving the smallest ball in the $n$-bracelet, since otherwise, the radius of $H_e$ and $H_f$ would need to be even larger. Let $H_a$ be the smallest ball and $H_b$ the ball connected to it by this beam. We can further assume that $H_b$ has the same radius $r$ as $H_a$. Let $d$ be the Euclidean distance between their centers. By Lemma \ref{lengthofg}, $g = \ln (\frac{d^2}{4r^2}) =  \ln(\frac{D^2}{4R^2})$.

Since balls cannot overlap, note that the centers of $H_e$ and $H_f$ must be a distance of at least $2\sqrt{Rr}$ from the centers of both $H_a$ and $H_b$. To minimize $R$, assume that $H_e$ is closer to $H_a$ than to $H_b$. We can then assume that $H_e$ and $H_f$ are both tangent to $H_a$ and by the Pythagorean Theorem, $4Rr \geq d^2/4 + D^2/4 = r^2 e^g +R^2e^g$. Hence, $R^2 - \frac{4Rr}{e^g} + r^2 \geq 0$. The quadratic formula then yields the result.

\end{proof}

We are now ready to apply the elder sibling property as a criterion for a vertical geodesic to be an unknotting tunnel.

\begin{theorem}
Let $M$ be a hyperbolic 3-manifold of one or two cusps with a vertical geodesic $\alpha$ that connects the cusps if there are two. If the ball-and-beam pattern for $(M, \alpha)$ is elder sibling, and $\alpha$ has length less than $\ln{(2)}$, then $\alpha$ is an unknotting tunnel.
\end{theorem}

\begin{proof}
The elder sibling property implies that  the ball-and-beam pattern for $(M, \alpha)$ is connected, and thus, by Lemma \ref{connectedimpliesnbracelet} it must contain an $n$-bracelet.  We will show that the ball-and-beam pattern also contains an $n$-disk, which is enough to prove that $\alpha$ is an unknotting tunnel.  

Since the ball-and-beam pattern for $(M, \alpha)$ contains an $n$-bracelet, we let $\beta$ be the curve in this $n$-bracelet that passes through the sequence of n balls cyclically connected by beams in order, and connects back to itself, and let $H_{min}$ be the smallest horoball in this $n$-bracelet.  Clearly $\beta$ is a nontrivial curve in the $n$-bracelet, so if it bounds a disk in the ball-and-beam pattern, the ball-and-beam pattern contains an $n$-disk and hence $\alpha$ is an unknotting tunnel.

If, on the other hand, $\beta$ does not bound a disk, it must be the case that the disk is blocked in at least one place by other balls and beams, as shown in Figure \ref{blockeddisk}. 


\begin{figure}[ht]
\centering
\includegraphics*[viewport=20 0 250 150]{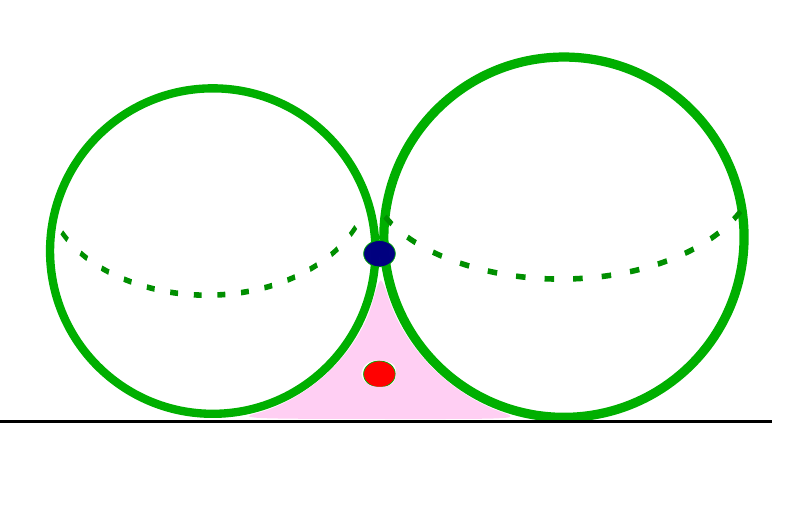}
\caption{A view of one place where the $n$-bracelet bounded by $\beta$ is blocked.  $\beta$ appears in crossection as a grey dot, the blocking beam as a black dot(assuming here the length of $\alpha$ is 0), and the blocking balls in grey.  The blocking disk is shaded.}
\label{oneblock}
\end{figure}

 Note that  $\beta$ has finite length in $H^3$, and thus is contained in a finite volume of $H^3$, so there can only be finitely many such places where the disk is blocked.  We can number them $1, 2,..,m$.  At each place where the disk is blocked, we must have $\beta$ passing through a \textit{blocking disk}, bounded by two horoballs and the beam connecting them, and the x-y plane, as shown in Figure \ref{oneblock}.  We call the beam that connects the two blocking balls and  runs along the boundary of the blocking disk the \textit{blocking beam}.  
 
 Notice that if the disk corresponding to $\beta$ is not blocked by a pair of horoballs connected by a beam, but instead just by a horoball that intersects the disk, we can deform the disk bounded by $\beta$ to pass around the blocking ball to eliminate the block.  At the $i^{th}$ blockage, we choose the larger  of the two horoballs whose connecting beam is blocking the disk and call it $H_i$.  For the bracelet to be blocked, $H_i$ must be at least $\frac{2+\sqrt{4-e^{2g}}}{e^g}$ times the size of the balls that it is reaching over, and hence it must have Euclidean height of at least $\frac{2+\sqrt{4-e^{2g}}}{e^g}$  times the Euclidean height of $H_{min}$.  If $g < \ln{2}$, then the quantity $\frac{2+\sqrt{4-e^{2g}}}{e^g}$ is greater than one, and so this blocking ball is larger than $H_{min}$. Furthermore, we can get from $H_i$ to $H_\infty$ via a chain of horoballs cyclically connected by beams such that the Euclidean height of each ball in this chain is larger than the Euclidean height of $H_i$, and so also larger than the Euclidean height of $H_{min}$.
 
Now, $\beta$ can be modified in the following way.  We will take the segment of $\beta$ that runs along $H_\infty$, and run it down the elder sibling chain of $H_1$, around the beam from $H_1$ that blocked our original $n$-bracelet, back up the elder sibling chain from $H_1$ to $H_\infty$, then down along the elder sibling chain of $H_2$, around the blocking beam incident to $H_2$, and back up to $H_\infty$, continuing this way for all $m$ blockages.  For an example of how this process might look for $m=1$ and a length-0 vertical geodesic, see Figure \ref{unblockeddisk} .  Lastly, we isotope the segments of our new $\beta$ that still run along $H_\infty$ so that these segments are directly above the segments of $\beta$ on the smaller horoballs, making the (possibly punctured) disk bounded by $\beta$ be completely vertical.

\begin{figure}[ht]
\centering
\includegraphics*[viewport=0 0 500 200]{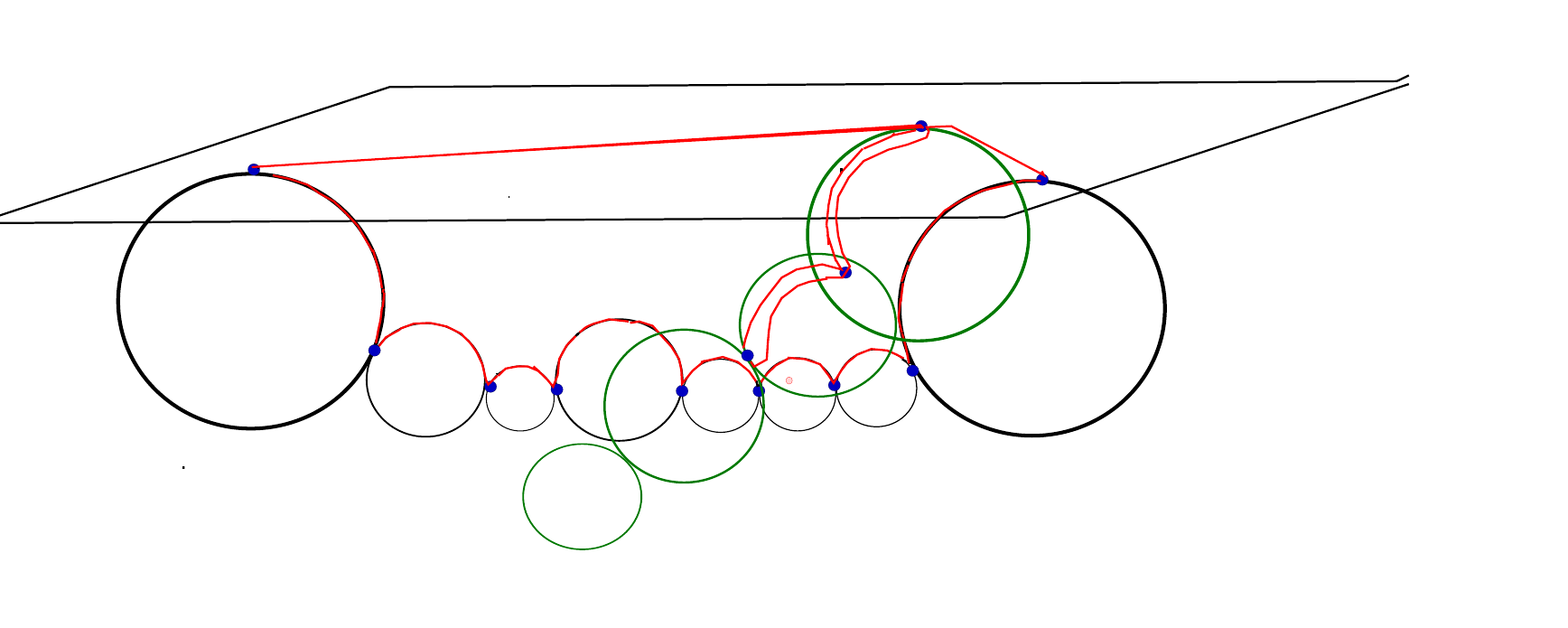}
\caption{We can modify $\beta$ in this way to eliminate a block.}
\label{unblockeddisk}
\end{figure}

We have created  a new $n$-bracelet, that is no longer blocked by our original $m$ blockages.  Let the new curve created by modifying $\beta$ in this way be called $\beta_2$, and note that $\beta_2$ is still a nontrivial curve in an $n$-bracelet.  Thus, if $\beta_2$ bounds a disk, then we have found an $n$-disk in the ball-and-beam pattern for $(M, \alpha)$ and so $\alpha$ must be an unknotting tunnel.  If not, then it must be the case that the disk bounded by $\beta_2$  is blocked, and so that $\beta_2$ passes through some finite number of blocking disks, as in Figure \ref{oneblock}.  These blocking balls incident to these blocking disks reach over the new elder sibling chains that $\beta_2$---but not $\beta$---runs along. These blocking horoballs correspond to new elder sibling chains extending up to $H_\infty$, each of whose smallest ball must be strictly larger than the smallest ball in the part of the bracelet that that they are reaching over.  Again, we repeat the process of modifying $\beta_2$, taking a segment of $\beta_2$ that runs along $H_\infty$, and running it down the elder sibling chain corresponding to each blockage, around $\alpha$, and back up to $H_\infty$, to create a new curve, which we call $\beta_3$.

It can be shown that this process cannot be continued indefinitely.  At each successive  modification of $\beta$, we involve elder-sibling chains whose smallest horoball is strictly larger than the smallest horoball in the elder sibling chains obtained in the previous step.  Thus, we obtain a sequence of horoballs of increasing size, all of which are larger than $H_{min}$.  Since there can be only finitely many horoballs in the cusp diagram larger than a given size, this sequence must be finite.  Thus, our process of modifying $\beta$ must end within a finite number of steps, and so $\beta_t$ will correspond to an unblocked $n$-bracelet for some finite $t$.  We have thus proven the presence of an $n$-disk in the ball-and-beam pattern for $(M,\alpha)$, and it follows that $\alpha$ is an unknotting tunnel.
\end{proof}

\bibliographystyle{plain}
\bibliography{SMALLrefs}

\begin{thebibliography}{99}
 
\bibitem{adams05}
 C. Adams,  Hyperbolic knots, 
 In \emph{Handbook of Knot Theory},
 3--18,
 Elsevier B. V., 
 2005.
 
\bibitem{adams95}
  C. Adams,
  Unknotting tunnels in hyperbolic 3-manifolds,
  Math. Ann.,
  \textbf{302},
  177--195
 (1995).
 
 \bibitem{AR}
 C. Adams and A. Reid,
 Unknotting tunnels in two-bridge knot and link complements,
  Comment. Math. Helv,
   \textbf{71},
   No. 4,
   617--627
   (1996).
   
   \bibitem{agol}
  I. Agol,
  Tameness of hyperbolic 3-manifolds,
  arXiv:math.GT/0405568


\bibitem{CG}
  D. Calegari and D. Gabai
  Shrinkwrapping and the taming of hyperbolic 3-manifolds,
  J. Amer. Math. Soc.
  \textbf{19},
  No. 2,
  385--446, 
  (2006).
  
  \bibitem{canary08}
  R. Canary,
  Marden's tameness conjecture: history and applications,
  \emph{Geometry, analysis, and topology of discrete groups}, 
  13--162,
  Adv, Lect. math (ALM), 6, 
  (2008).


 \bibitem{CLP}
  D. Cooper, M. Lackenby and J. Purcell,
 The length of unknotting tunnels. 
 Algebr. Geom. Topol. 
 \textbf{10},
  No. 2,
   637--661,
   (2010).
   
   
   \bibitem{CFP}
   D. Cooper, D. Futer and J. Purcell,
   Dehn filling and the geometry of unknotting tunnels,
   ArXiv:1105-3461,
   (2011).
   
    \bibitem{FM}
  M.  Freedman and C. McMullen,
  Elder siblings and the taming of hyperbolic 3-manifolds,
  Ann. Acad. Sci. Fenn. Math.,
  \textbf{23},
  No. 2, 
  415--428
  (1998).

\bibitem{Marden}
A. Marden, 
The geometry of finitely generated kleinian groups,
 Ann. of Math.,
 \textbf{99}
 No. 2,
  383--462,
 (1974). 


   
   \bibitem{mostow}
   G. D. Mostow,
   Quasi-conformal mappings in n-space and the rigidity of the hyperbolic space forms,
   Publ. Math. IHES,
   \textbf{43}
   (1968).
   
   \bibitem{snappea}
  J. Weeks
  SnapPea, A computer program for created and studying hyperbolic 3-manifolds,
  available at http://www.geometrygames.org/SnapPea
   
   
\end{thebibliography}

\end{document}